\documentclass{amsart}
\usepackage{graphicx} % Required for inserting images

\usepackage{fullpage}

\usepackage{graphicx} % Required for inserting images
\usepackage[utf8]{inputenc}
\usepackage[english]{babel}
\usepackage{csquotes}
\usepackage[T1]{fontenc}
\usepackage{amssymb}
\usepackage{mathtools}
\usepackage{amsfonts}
\usepackage{amsmath}
\usepackage{amsthm}
\usepackage{mathrsfs}
\usepackage{tikz}
\usepackage{tikz-cd}
\usepackage{changepage}
\usepackage{url}
\graphicspath{{images/}}
\usepackage{fancyhdr}
\usepackage{amsmath}
\usepackage{amssymb}
\usepackage{amsfonts}
\usepackage{amsthm}

\usepackage{hyperref}
\hypersetup{
	colorlinks=true,
	linkcolor=blue,
	citecolor=blue,
	urlcolor=blue,
	pdftitle={}
}

\DeclareMathOperator{\diam}{diam}
\DeclareMathOperator{\Isom}{Isom}

\usepackage[backend=biber,style=numeric]{biblatex}
\newtheorem{theorem}{Theorem}  
\newtheorem{cor}[theorem]{Corollary}% Corollary in the Intro
\theoremstyle{plain}
\newtheorem{thm}{Theorem}
\newtheorem{lemma}[thm]{Lemma}
\newtheorem{corollary}[thm]{Corollary}

\newtheorem{prop}[thm]{Proposition}

\theoremstyle{definition}
\newtheorem{definition}[thm]{Definition}

\newtheorem{remark}[thm]{Remark}
\newtheorem{exmp}[thm]{Example}

\numberwithin{equation}{section}
\numberwithin{thm}{section}
\usepackage{setspace}
\theoremstyle{remark} 
\newtheorem*{ack}{Acknowledgements}

\title{Stability and Finiteness of Wasserstein Spaces }
\date{\today}

\author[Mohammad Al-Attar]{Mohammad Alattar}
\address[Al-Attar]{Department of Mathematical Sciences, Durham University, United Kingdom}
\email{\href{mailto:mohammad.al-attar@durham.ac.uk}{mohammad.al-attar@durham.ac.uk}}

\begin{document}

\begin{abstract}
   Under Gromov--Hausdorff convergence, and equivariant Gromov--Hausdorff convergence, we prove stability results of Wasserstein spaces over certain classes of singular and non-singular spaces. For example, we obtain an analogue of Perelman's stability theorem on Wasserstein spaces.
    
\end{abstract}

    \subjclass[2020]{53C23, 53C20, 51K10}
	\keywords{equivariant Gromov--Hausdorff convergence, Gromov--Hausdorff convergence, Wasserstein spaces, stability}

\maketitle

\section{Introduction}
Finiteness theorems for finite dimensional spaces with curvature bounds can be traced back to 1967,  with Weinstein \cite{Weinstein-Homotopy-Finiteness} proving that given any even natural number $n$, and any $\delta>0$, there are only finitely many homotopy types of $n$-dimensional, $\delta$-pinched manifolds. Shortly after, in 1970, Cheeger obtained several finiteness theorems in his celebrated work \cite{Jeff-Cheeger}. For instance, he proved that given any natural number $n$, and real numbers $D,v,k>0$, there are only finitely many diffeomorphism types of compact $n$-dimensional manifolds $X$ admitting a Riemannian metric such that diameter$(X)$ $\leq D$, vol$(X)$ $\geq v$ and $|\sec(X)|\leq k^2$, where $\sec(X)$ denotes the sectional curvature. These results served as departure points for further seminal and important developments in Riemannian geometry (see for instance \cite{Anderson--Cheeger, Homotopy-Finiteness-GP, FinitenessViaControlledTopology,Grove-Finiteness-Theorems-Survey,perelman1991alexandrov}).

In 1981, Gromov \cite{GH-Convergence} introduced a powerful notion of convergence, now known as \emph{Gromov--Hausdorff convergence}, that has been useful in establishing (among other things) finiteness theorems through the notion of ``stability''. For instance, using Gromov--Hausdorff convergence, Grove, Petersen and Wu \cite{FinitenessViaControlledTopology}, in 1990, relaxed the curvature condition in Cheeger's finiteness theorem and showed that for any $n\neq 3$, real number $k$, and any positive numbers $D$ and $v$, the class consisting of compact $n$-dimensional Riemannian manifolds $X$ with diameter$(X)\leq D$,  $\sec(X)\geq k$, and vol$(X)\geq v$ contains finitely many homeomorphism types.

In 1991, Perelman \cite{perelman1991alexandrov,Vitali}, using Gromov--Hausdorff convergence, showed that given a non-collapsing convergent sequence of compact Alexandrov spaces, with a uniform lower curvature bound, uniform finite dimension and uniform upper diameter bound, the tail of the sequence \textit{stabilizes} topologically. That is, all spaces in the tail of the sequence are mutually homeomorphic. As a corollary, one obtains homeomorphism finiteness (in all dimensions) for a more general class of spaces that need not be manifolds.

In all cases of finiteness and stability, attention has been restricted to finite dimensional spaces. Thus, the question of stability in infinite dimensions remains open. In this paper, we focus our attention on  \emph{Wasserstein spaces}, which are infinite-dimensional, and prove stability and finiteness results for these spaces in different settings. Our proofs are not technical, and these results are, to the best of our knowledge, the first stability and finiteness results for infinite-dimensional spaces.

 Let $\mathcal{A}^{\leq}(n,K,D)$ denote the class of compact Alexandrov spaces with dimension at most $n$,  curvature bounded below by $K$, and diameter bounded above by $D$.  We denote by $\mathcal{A}(n,K,D)$ all Alexandrov spaces in $\mathcal{A}^{\leq}(n,K,D)$ with dimension $n$. Given $v>0$, we will denote by $\mathcal{A}(n,K,D,v)$ the subset of $\mathcal{A}(n,K,D)$ where the Alexandrov spaces under consideration have volume bounded below by $v$. Furthermore, given a sequence of compact metric spaces $\{X_i\}_{i\in \mathbb{N}}$ and a compact metric space $X$, we will denote by $X_i\rightarrow_{GH}X$ the Gromov--Hausdorff convergence of $X_i$ to $X$. We will denote by $\mathbb{P}_2(X)$ the Wasserstein space over $X$ equipped with the $2$-Wasserstein metric (see Definition \ref{wasserstein metric}).

\begin{theorem}
\label{theorem A}

Assume $\{X_i\}_{i\in \mathbb{N}}$ and $X$ are in $\mathcal{A}^{\leq}(n,K,D)$. Then the following assertions are equivalent:

\begin{enumerate}
    \item $\mathbb{P}_2(X_i)\rightarrow_{GH} \mathbb{P}_2(X)$.
    \item $X_i\rightarrow_{GH}X$.
    \item $(\mathbb{P}_2(X_i),X_i)\rightarrow_{GH}(\mathbb{P}_2(X),X)$.
\end{enumerate}
\end{theorem}

We point out a few remarks concerning the above theorem. First, the interesting implication is $(1)\implies (2)$. The implication $(2)\implies(1)$ is well known (see Corollary 4.3 in \cite{LottandVillani}) and holds for general compact metric spaces. Second, the notion of convergence in $(3)$ is known as \emph{Gromov--Hausdorff convergence of metric pairs}. It was introduced in \cite{che2022metric}, and later studied extensively in \cite{ahmudaandche}. This notion of convergence generalizes Gromov--Hausdorff convergence and takes into account pairs of the form $(X,A)$, where $X$ is a metric space, and $A$ is a closed non-empty subset of $X$. In our theorem, we identify the base space $X$ with the set of Dirac deltas in the Wasserstein space $\mathbb{P}_2(X)$. Thus, $(2)\implies(3)$ becomes trivial. However, the new idea for this implication is that we give an equivalent characterization of convergence of metric pairs using the notion of $\epsilon$-isometries (see Proposition~\ref{new characterization}). This equivalent characterization will be useful when we introduce the notion of relative equivariant Gromov--Hausdorff convergence (see Definition~\ref{relative gromov hausdorff convergence}) which takes into account a triple $(X,A,G)$, where $X$ is a compact metric space, $A$ is a closed, non-empty $G$-invariant subset of $X$ and $G$ is a closed subgroup of $\Isom(X)$, the group of isometries of $X$.

A consequence of the above theorem, is the following analogue of Perelman's stability theorem for Wasserstein spaces.

\begin{cor}
\label{corollary1}
Let $\{X_i\}_{i\in \mathbb{N}}$ and $X$ be spaces in $\mathcal{A}(n,K,D)$. If $\mathbb{P}_2(X_i)\rightarrow_{GH}\mathbb{P}_2(X)$ then for all sufficiently large $i$, $\mathbb{P}_2(X_i)$ and $\mathbb{P}_2(X)$ are homeomorphic, and the homeomorphisms can be taken to be Gromov--Hausdorff approximations.

\end{cor}

Note that in the preceding corollary we do not impose any other structure on the Wasserstein spaces other than that coming from the base spaces.

We further have the following finiteness result. For any $n\in \mathbb{N}$, $K\in \mathbb{R}, D>0$ and $v>0$, let $W_2(n,K,D,v)$ denote the class consisting of Wasserstein spaces with base spaces in $\mathcal{A}(n,K,D,v)$ and equipped with the $2$-Wasserstein metric.

\begin{cor}
\label{corollary 2}
There are only finitely many topological types in $W_2(n,K,D,v)$.

\end{cor}

We note that the technique in this paper can be applied  whenever we have a stability result on the base spaces and essentially reduces to the following result.

\begin{theorem}
\label{general theorem}

Let $\mathfrak{X}$ denote a precompact class (with respect to the Gromov--Hausdorff topology) of compact non-branching metric measure spaces having good transport behavior $(GTB)$ and their reference measures are qualitivatively non-degenerate. Assume that if $\{X_i\}_{i=1}^{\infty}$ is a sequence in $\mathfrak{X}$, and $X_i\rightarrow_{GH} X$, then $X$ is a non-branching compact metric space admitting a measure that is qualitatively non-degenerate and has good transport 
behavior (we will take such a measure as a reference measure on $X$). Then, the following assertions are equivalent:

\begin{enumerate}
    \item $\mathbb{P}_2(X_i)\rightarrow_{GH}\mathbb{P}_2(X)$.
    \item $X_i\rightarrow_{GH} X$.
\end{enumerate}

\end{theorem}

The notion of good transport behavior was first introduced in the paper by Galaz-Garc\'ia, Kell, Mondino, and Sosa \cite{Garcia-Kell-Mondino}, and further studied by Kell \cite{Kell-Studied}. The class of metric measure spaces satisfying such condition is large. For instance, according to \cite{Garcia-Kell-Mondino}, strong $\mathrm{CD}^{*}(K,N)$ and essentially non-branching $\mathrm{MCP}(K,N)$ spaces admit good transport behavior. The notion of a measure being qualitatively non-degenerate was introduced by Cavalletti and Huesmann \cite{Cavalletti-Husemann} and later studied by Kell \cite{Kell-Studied}. It is technical to state, and we refer the reader to \cite{Kell-Studied,Cavalletti-Husemann,santos2021isometric} for more details, but we note a measure being qualitatively non-degenerate yields desirable information. For example, a qualitatively non-degenerate measure is doubling (Lemma 2.8 in \cite{santos2021isometric}).

We note  that the classes $\mathcal{A}^{\leq}(n,K,D)$, $\mathcal{A}(n,K,D)$ satisfy the hypotheses of Theorem~\ref{general theorem}. However, the properties in Theorem \ref{general theorem} are not closed under Gromov--Hausdorff convergence in general. This is because a compact metric measure space endowed with a qualitatively non-degenerate measure is doubling and thus must have finite dimension.

Using almost verbatim the proofs of Theorems \ref{general theorem} and \ref{theorem A}, and noting that $\mathrm{RCD}(K,N)$ spaces are non-branching \cite{qin}, we obtain the following corollary.

\begin{cor}
\label{RCD corollary}
Assume each $(X_i,d_i,\mathfrak{m}_i)$ is an compact $\mathrm{RCD}(K,N)$ space with uniform upper diameter bound. Assume $(X,d)$ is an $N$-dimensional compact Riemannian manifold. If  $\mathbb{P}_2(X_i)\rightarrow_{GH} \mathbb{P}_2(X)$ then for all large $i$, there exists a homeomorphism $G_i\colon \mathbb{P}_2(X_i)\rightarrow \mathbb{P}_2(X)$ that is Lipschitz and its inverse is Hölder.
\end{cor}

A consequence of our next main theorem is an equivariant stability result for Wasserstein spaces. We first discuss notation.  Given a closed subgroup $G$ of $\Isom(X)$, where $X$ is a compact metric space, we denote by $G^{\#}$ the induced subgroup of $\Isom(\mathbb{P}_2(X))$ given by the push-forward of maps. Given a sequence of pairs $\{(X_i,G_i)\}_{i\in \mathbb{N}}$, where each $X_i$ is a compact metric space and $G_i$ is a closed subgroup of $\Isom(X)$, we will denote, by $(X_i,G_i)\rightarrow_{eGH}(X,G)$ the equivariant Gromov--Hausdorff convergence of $(X_i,G_i)$ to $(X,G)$. This convergence, originally introduced by Fukaya (see \cite{fukaya1986theory,fukayaandyamaguchiannals,hausdorffconvergencefukaya}), 
has been useful in establishing interesting results in both the singular and non-singular setting (see for instance \cite{zamora2022betti,corro2024collapsing,brue2023stability,harvey2016convergence,Zamorafiniteness, corro2024cohomogeneity,cavallucci2023ghcompactification}). Furthermore,  it has been useful in obtaining finiteness and stability results in the singular setting. For instance, using equivariant Gromov--Hausdorff convergence, Zamora extended Anderson finiteness to the $\mathrm{RCD}^{*}(K,N)$ setting \cite{Zamorafiniteness}.

\begin{theorem}
\label{theorem E}

Let $\{X_i\}_{i\in \mathbb{N}}$ and $X$ be closed Riemannian manifolds with uniform lower sectional curvature bound $K>0$, uniform upper diameter bound and uniform upper dimension bound. Assume $H_i$ and $H$ are closed subgroups of $\Isom(\mathbb{P}_2(X_i))$ and $\Isom(\mathbb{P}_2(X))$ respectively. Then the following assertions are equivalent:

\begin{enumerate}
    \item $(\mathbb{P}_2(X_i),H_i)\rightarrow_{eGH}(\mathbb{P}_2(X),H)$.
    \item There exists unique closed subgroups $G_i$ of $\Isom(X)$ and $G$ of $\Isom(X)$ such that $G_i^{\#}=H_i$ and $G^{\#}=H$ and such that $(X_i,G_i)\rightarrow_{eGH}(X,G)$.
    \item $(\mathbb{P}_2(X_i),X_i,H_i)\rightarrow_{eGH}(\mathbb{P}_2(X),X,H)$.

\end{enumerate}

\end{theorem}

We emphasize that we identify the base space with the (closed) subset of Dirac deltas in the Wasserstein space. In the presence of a uniform lower curvature bound, we have the following corollary, an equivariant stability result on Wasserstein spaces.

\begin{cor}
\label{corollary 3}
Let $\{X_i\}_{i\in \mathbb{N}}$ and $X$ be closed Riemannian manifolds with uniform lower sectional curvature bound $K>0$, uniform upper diameter bound and uniform dimension. Assume $H_i$ and $H$ are closed subgroups of $\Isom(\mathbb{P}_2(X_i))$ and $\Isom(\mathbb{P}_2(X))$ respectively of the same dimension. If $(\mathbb{P}_2(X_i),H_i)\rightarrow_{eGH}(\mathbb{P}_2(X),H)$, then for all large $i$, there exists a homeomorphism $f_i\colon\mathbb{P}_2(X_i)\rightarrow \mathbb{P}_2(X)$ that is compatible with some group isomorphism $\chi_i\colon H_i\rightarrow H$.
\end{cor}

Our article is organized as follows. In section \ref{preliminaries section}, we discuss the preliminaries that will be useful in this paper. Moreover, we will give a new characterization of convergence of metric pairs. In section \ref{proof section} we give proofs of the theorems and corollaries.

% ACKNOWLEDGMENTS 

\begin{ack}
 It is a pleasure to thank Fernando Galaz--Garc\'ia, Martin Kerin, Kohei Suzuki, Jaime Santos-Rodr\'iguez, and Mauricio Che for helpful conversations and comments during the Durham Metric Geometry Reading Seminar. I would also like to thank Sergio Zamora and Yanpeng Zhi for several interesting conversations and Aseel alnajjar for several comments on the first draft of this manuscript.
Part of this work was completed at the Erwin Schrödinger International Institute for Mathematics and Physics (ESI), where the author received financial support to participate in the workshop \emph{Synthetic Curvature Bounds for Non-Smooth Spaces: Beyond Finite Dimension}. The author would like to thank ESI and the conference organizers for the excellent atmosphere and working conditions.   
\end{ack}

\section{Preliminaries}\label{preliminaries section}

In this section, we collect preliminary definitions and results we will use in the proof of our main results.

\subsection{Convergence of Metric Pairs and Equivariant Convergence of Metric Pairs}

In this section, we give a new equivalent characterization of the notion of convergence of metric pairs which will allow us to introduce, canonically, the notion of relative equivariant Gromov--Hausdorff convergence. The equivalent characterization is demonstrated in the following proposition.  We refer the reader to the paper by Ahumada G\'omez and Che \cite{ahmudaandche} for results concerning the (non-equivariant) convergence of metric pairs.

\begin{prop}
\label{new characterization}

Let $\{X_i\}_{i\in \mathbb{N}}$ be a sequence of compact metric spaces. Let $\{A_i\}_{i\in \mathbb{N}}$ be a collection of closed non-empty subspaces  $A_i\subseteq X_i$. Let $X$ be a compact metric space, and $A\subseteq X$ a closed non-empty subset. Then the following assertions are equivalent:

\begin{enumerate}
    \item $(X_i,A_i)\rightarrow_{GH} (X,A)$.
    \item There exists a sequence of positive numbers $\{\epsilon_i\}_{i\in \mathbb{N}}$ such that $\epsilon_i\rightarrow 0$ and $\epsilon_i$-Gromov--Hausdorff approximations $f_i\colon X_i\rightarrow X$, $g_i\colon A_i\rightarrow A$ such that $f_i$ and $g_i$ are $\epsilon_i$-close, i.e., $d_{X}(f_i(x_i),g_i(x_i))\leq\epsilon_i$ for all $x_i\in A_i$.
\end{enumerate}

\end{prop}

\begin{proof}

Assume $(X_i,A_i)\rightarrow_{GH}(X,A)$. Then there exists a sequence of numbers $\{\epsilon_i\}_{i\in \mathbb{N}}$ that tends to $0$ and $\epsilon_i$-Gromov--Hausdorff approximations $f_i\colon X_i\rightarrow X$ such that $d_{H}(f_i(A_i),A)< \epsilon_i$. Here, $d_{H}$ denotes the Hausdorff distance.  We define $g_i\colon A_i\rightarrow A$ as follows. Since $d_{H}(f_i(A_i),A)<\epsilon_i$, for each $a_i\in A_i$ choose $g_i(a_i)\in A$ such that $d_{X}(f_i(a_i),g_i(a_i))<\epsilon_i$. We will show that the map $g_i\colon A_i\rightarrow A$ is a $3\epsilon_i$-Gromov--Hausdorff approximation. Indeed, for each $a\in A$, there exists an $a_i\in A_i$ such that $d_{X}(f_i(a_i),a)\leq \epsilon$. Therefore, by the triangle inequality, we get $d_{X}(g_i(a_i),a)\leq 3\epsilon_i$. Now we show that $g_i$ is almost distance preserving up to an error $3\epsilon_i$. We only show one side of the inequality, for the other is similar. Let $a_i,a_i'\in A_i$. Then 
\[d_{X}(g_n(a_i),g_n(a_i'))\leq d_{X}(g_i(a_i),f_i(a_i))+d_{X}(f_i(a_i),f_i(a_i'))+d_{X}(f_i(a_i'),g_i(a_i')).\]

Now note that $d_{X}(g_i(a_i),f_i(a_i))\leq \epsilon_i$ and $d_{X}(f_i(a_i'),g_i(a_i'))\leq \epsilon_i$. Thus, as $f_i$ is an $\epsilon_i$-Gromov--Hausdorff approximation, we get 
\[ d_{X}(g_i(a_i),g_i(a_i'))\leq 3\epsilon_i+ d_{X_i}(a_i,a_i').\]

Now we show the reverse direction. Define $\xi_i\colon X_i\rightarrow X$ by 
\[ \xi_i(x_i)=\begin{cases} 
      f_i(x_i) & x_i\in X_i\backslash A_i \\
      g_i(x_i) & x_i \in A_i.\\
\end{cases}
\]

Since $g_i\colon A_i\rightarrow A$ is an $\epsilon_i$-Gromov--Hausdorff approximation, it follows that $d_{H}(g_i(A_i),A)\leq \epsilon_i$. It remains to verify that $\xi_i$ is a Gromov--Hausdorff approximation. It is clear that $\xi_i$ is $\epsilon_i$-surjective. Therefore, it remains to prove that $\xi_i$ is almost distance preserving up to an error that tends to $0$. Once again, we only verify one direction of the inequality, since the other is similar. Let $x_i,x_i'\in X_i$. It is clear that it suffices to assume that $x_i\in X_i\backslash A_i$ and that $x_i'\in A_i$. Then
\[d_{X}(\xi_i(x_i),\xi_i(x_i'))=d_{X}(f_i(x_i),g_i(x_i'))\leq d_{X}(f_i(x_i),f_i(x_i'))+d_{X}(f_i(x_i'),g_i(x_i')).\]

Since $f_i$ is an $\epsilon_i$-Gromov--Hausdorff approximation, $d_{X}(f_i(x_i),f_i(x_i'))\leq \epsilon_i+d_{X_i}(x_i,x_i')$. Since $f_i$ and $g_i$ are $\epsilon_i$-close, $d_{X}(f_i(x_i'),g_i(x_i'))\leq \epsilon_i$.  
\end{proof}

The proposition above suggests the following definition of relative equivariant Gromov--Hausdorff convergence.

\begin{definition}
\label{relative gromov hausdorff convergence}

Let $X$ and $Y$ be compact metric spaces. Suppose $G_{X}$ and $G_{Y}$ are closed subgroups of $\Isom(X)$ and $\Isom(Y)$ respectively. Assume $A_{X}$ is a closed non-empty $G_{X}$-invariant subset of $X$, and $A_{Y}$ is a closed non-empty $G_{Y}$-invariant subset of $Y$. For $\epsilon>0$, we define an \emph{$\epsilon$-equivariant Gromov--Hausdorff approximation between $(X,A_{X},G_{X})$ and $(Y,A_{Y},G_{Y})$} to be a quadruple of maps
\[ (f\colon X\rightarrow Y,f'\colon A_{X}\rightarrow A_{Y}, \theta\colon G_{X}\rightarrow G_{Y},\psi\colon G_{Y}\rightarrow G_{X})\]

subject to the following conditions:

\begin{enumerate}
    \item The triple $(f,\theta,\psi)$ is an $\epsilon$-equivariant Gromov--Hausdorff approximation (in the usual sense).
    \item $f'$ is an $\epsilon$-Gromov--Hausdorff approximation.
    \item The maps $f$ and $f'$ are $\epsilon$-close.
\end{enumerate}

\end{definition}

\begin{remark}
We use the terms ``equivariant Gromov--Hausdorff approximations'' and ``equivariant approximations'' interchangeably.

\end{remark}

\begin{definition}
Given a sequence of triples $\{(X_i,A_i,G_i)\}_{i\in \mathbb{N}}$, where $X_i$ is a compact metric space, $G_i$ is a closed subgoup of $\Isom(X_i)$ and $A_i$ is a closed non-empty $G_i$-invariant subset of $X_i$, we say that $(X_i,A_i,G_i)$ \textit{Gromov--Hausdorff equivariantly converges to} $(X,A,G)$, denoted $(X_i,A_i,G_i)\rightarrow_{eGH}(X,A,G)$, where $X$ is a compact metric space, $G$ is a closed subgroup of Isom$(X)$ and $A$ is a closed subset of $X$ that is $G$-invariant, \textit{if} there exists a sequence of positive real numbers $\{\epsilon_i\}_{i\in \mathbb{N}}$ such that $\epsilon_i\rightarrow 0$  and $\epsilon_i$-equivariant Gromov--Hausdorff approximations $(f_i,f'_i,\theta_i,\psi_i)\colon (X_i,A_i,G_i)\rightarrow (X,A,G)$.
\end{definition}

\subsection{Further definitions and auxiliary results}
For the convenience of the reader, we include the definitions of \emph{Wasserstein spaces}, \emph{good transport behavior} and \emph{isometric rigidity}.

\begin{definition}
\label{wasserstein metric}

Let $(X,d)$  be a Polish metric space and let $p\in [1,\infty)$.  Given any two probability measure $\mu$ and $\nu$ on $X$, we define the \emph{$L^p$-Wasserstein metric} between $\mu$ and $\nu$, denoted by $W_p$, to be 
\[W_p^p(\mu,\nu) =\inf_{\pi} \int d(x,y)^pd\pi(x,y),\]
where the infimum is taken over all admissible measures $\pi$ having marginals $\mu$ and $\nu$. The $L^p$-Wasserstein metric is a metric on the space of probability measures with finite $p$-moments, which is denoted by $\mathbb{P}_p(X)$.
\end{definition}

\begin{definition}A metric measure space $(X,d,\mathfrak{m})$ is said to have \emph{good transport behavior} if, for any two probability measures $\mu,\nu$ in $\mathbb{P}_2(X)$, where $\mu \ll \mathfrak{m}$, any optimal transport map between them is induced by a map. 
\end{definition}

We motivate the concept of isometric rigidity by the following example.

\begin{exmp}
If $X$ is a compact metric space and $f:X\rightarrow X$ is an isometry,  then $f$ induces an isometry $f_{\#}:\mathbb{P}_2(X)\rightarrow \mathbb{P}_2(X)$.

\end{exmp}

Thus, it is natural to formulate the following definition.

\begin{definition}
Let $(X,d)$ be a metric space. Then, $X$ is said to be \textit{isometrically rigid} if given any isometry $\theta:\mathbb{P}_2(X)\rightarrow \mathbb{P}_2(X)$, there exists an isometry $f:X\rightarrow X$ such that $\theta = f_{\#}$. In particular, every isometry of $\mathbb{P}_2(X)$ is induced by an isometry of $X$; the base space.

\end{definition}

The concept of isometric rigidity was first introduced by Kloeckner in the context of euclidean space in \cite{benoit1}, where he showed, among other things, that, $\mathbb{P}_2(\mathbb{R})$ is not isometrically rigid. Later, Bertrand and Kloeckner \cite{bertrantandklockner} studied the concept of isometric rigidity and showed that if $X$ is a negatively curved geodesically complete Hadamard space, then $\mathbb{P}_2(X)$ is isometrically rigid. In 2021, 
Santos Rodriguez showed \cite{santos2021isometric} that if $X$ is a positively curved closed Riemannian manifold, then $X$ is isometrically rigid and that, if $X$ is a CROSS, then, in fact, the isometry groups of $\mathbb{P}_p(X)$ and $X$ are the same for all $p>1$.  In the same paper,  Santos Rodriguez showed the following result which we will use frequently in this paper (Corollary 3.8, \cite{santos2021isometric}).

\begin{prop}
\label{Jaime's Result}
Let $(X,d_{X},\mathfrak{m}_{X})$ and $(Y,d_{Y},\mathfrak{m}_{Y})$ be two compact non-branching metric measure spaces equipped with qualitatively non-degenerate measures and such that they have good transport behavior. Suppose that there exists an isometry $\Phi:\mathbb{P}_2(X)\rightarrow \mathbb{P}_2(Y)$. Then $(X,d_{X})$ and $(Y,d_{Y})$ are isometric.

\end{prop}

In fact, Santos Rodriguez in \cite{santos2021isometric} showed the above proposition for Wasserstein spaces with finite $p$-moments, where $p\in (1,\infty)$. 

\section{Proofs}\label{proof section}

In what follows, we assume that the maps are sufficiently regular. Indeed, this is possible because if $X$ is a compact metric space then $\Isom(X)$ is compact with the compact-open topology. Moreover, the compact open topology on $\Isom(X)$ can be metrized with the uniform metric.  One can always replace $\epsilon$-isometries between compact metric spaces by almost-isometries that are measurable.  In particular, we have the following simple lemma.

\begin{lemma}
Assume $X$ and $Y$ are compact metric spaces and $G_{X}$ and $G_{Y}$ are closed subgroups of $\Isom(X)$ and $\Isom(Y)$ respectively. Equip $G_{X}$ and $G_{Y}$ with the uniform metrics $d_{G_X}$ and $d_{G_Y}$ respectively. If $(f\colon X\rightarrow Y,\theta\colon G_{X}\rightarrow G_{Y},\psi\colon G_{Y}\rightarrow G_{X})$ is an $\epsilon$-equivariant approximation, then, $f$, $\theta$ and $\psi$ can be chosen to be measurable. 
\end{lemma}

\begin{proof}
By Corollary 3.4 in \cite{alattar2023stability}, it follows that $\theta$ is an $5\epsilon$-approximation. Furthermore, observe that $\psi$ is an almost inverse to $\theta$. In particular, for $\lambda\in G_{Y}$, and $g\in G_{X}$, one has $d_{G_Y}(\theta(\psi(\lambda)),\lambda)\leq 4\epsilon$ and $d_{G_{X}}(\psi(\theta(g)),g)\leq 4\epsilon$. Hence, as shown in Lemma 4.1 in \cite{Chung},  one can obtain measurable approximations $f_1\colon X\rightarrow Y$, $\theta_1\colon G_{X}\rightarrow G_{Y}$ and $\psi_1\colon G_{Y}\rightarrow G_{X}$ that remain close to $f$, $\theta$ and $\psi$ respectively. Hence the result follows.
\end{proof}

\subsection*{Proof of Theorem \ref{general theorem}}
The implication $(2)\implies (1)$ is well known (see Corollary 4.3 in \cite{LottandVillani}). Let us show $(1)\implies (2)$. Assume for the sake of obtaining a contradiction that $\mathbb{P}_2(X_i)\rightarrow_{GH}\mathbb{P}_2(X)$ holds but $X_i\rightarrow_{GH}X$ does not hold. Hence, there exists a $\delta>0$ such that (up to a subsequence), $d_{GH}(X_i,X)\geq \delta >0$ for all $i$. Since the class $\mathfrak{X}$ is pre-compact, and by our assumptions, it follows that up to a further subsequence, $X_i\rightarrow_{GH} X'$, where $X'$ is a compact non-branching metric space. Hence, $\mathbb{P}_2(X_i)\rightarrow_{GH}\mathbb{P}_2(X')$. By assumption, $\mathbb{P}_2(X)$ and $\mathbb{P}_2(X')$ are isometric. Now, by Proposition \ref{Jaime's Result}, $X$ and $X'$ are therefore isometric, which is a contradiction. \qed

\subsection*{Proof of Theorem \ref{theorem A}} The proof is similar to the proof of Theorem \ref{general theorem}. However, we give details for the convenience of the reader.
The implications $(2)\implies (1)$ and $(3)\implies (1)$ are clear. The implication $(2)\implies (3)$ follows at once from the following observation. Let $f_i \colon X_i\rightarrow X$ be $\epsilon_i$-approximations by measurable maps. Then each $f_i$ induces an $\tilde{\epsilon_i}$-approximation $(f_i)_{\#}\colon\mathbb{P}_2(X_i)\rightarrow \mathbb{P}_2(X)$, with $\tilde{\epsilon_i}\rightarrow 0$, and such that, on Dirac deltas, $(f_i)_{\#}$ is a Hausdorff approximation. Hence, one may use Proposition \ref{new characterization} to establish this implication. We shall now show $(1)\implies (2)$.

We must prove that $X_i\rightarrow_{GH} X$. Assume otherwise. Hence, there exists $\delta>0$ such that, up to a subsequence, $d_{GH}(X_i,X)\geq \delta>0$ for all $i$. Here, $d_{GH}$ denotes the Gromov--Hausdorff distance. Since $\mathcal{A}^{\leq}(n,K,D)$ is, with respect to $d_{GH}$, compact  \cite{burago1992ad}, there exists a compact Alexandrov space $X'$ such that, up to a further subsequence, $X_i\rightarrow_{GH}X'$.  Hence, $\mathbb{P}_2(X_i)\rightarrow_{GH}\mathbb{P}_2(X')$. By assumption, $\mathbb{P}_2(X_i)\rightarrow_{GH}\mathbb{P}_2(X)$. Therefore, $\mathbb{P}_2(X')$ and $\mathbb{P}_2(X)$ are isometric. Alexandrov spaces have good transport behavior, and on them, the Hausdorff measure is qualitatively non-degenerate, therefore $X$ and $X'$ are isometric. Hence we have $X_i\rightarrow_{GH}X$.  
\qed

\subsection*{Proof of Corollary~\ref{corollary1}} The result follows from Perelman's stability Theorem. Note that if $f\colon X\rightarrow Y$ is a map, then $f_{\#}$ is continuous if and only if $f$ is continuous (see Remark 4.14 in \cite{villani2009optimal}).
\qed
%Let us now prove Corollary \ref{corollary 2}.

\subsection*{Proof of Corollary \ref{corollary 2}}
We proceed by contradiction. Assume we have a sequence $(\mathbb{P}_2(X_i))_{i\in \mathbb{N}}$ of topologically pairwise inequivalent Wasserstein spaces in $W_2(n,K,D,v)$. Then, since $\mathcal{A}(n,K,D,v)$ is precompact and we have volume bounded below by $v>0$, there exists a compact $n$-dimensional Alexandrov space $X$ such that (up to a subsequence), $X_i\rightarrow_{GH}X$. Hence, $\mathbb{P}_2(X_i)\rightarrow_{GH} \mathbb{P}_2(X)$. Therefore, by Corollary \ref{corollary1}, for all large $i$, $\mathbb{P}_2(X_i)$ and $\mathbb{P}_2(X)$ are homeomorphic, which is a contradiction.

\qed

\subsection*{Proof of Corollary \ref{RCD corollary}}
This follows in an analogous manner as in the proof of Corollary \ref{corollary1}. However, one uses the result in \cite{HondaandPeng} instead of Perelman's stability theorem.
\qed
\\

%\section*{Proof of Theorem \ref{theorem E}}

Now we will prove Theorem \ref{theorem E}. First, we need two lemmas. 

\begin{lemma}
\label{pro: inducing-equivariant-convergence-wasserstein-spaces}
Let $X$ and $Y$ be compact metric spaces and $G_{X}$ and $G_{Y}$ are closed subgroups of $\Isom(X)$ and $\Isom(Y)$ respectively. An $\epsilon$-equivariant approximation, $(f,\theta,\psi)$ between $(X,G_{X})$ and $(Y,G_{Y})$  induces an $\tilde{\epsilon}$-equivariant approximation $(f_{\#},\theta_{\#},\psi_{\#})$ between $(\mathbb{P}_2(X),G_{X}^{\#})$ and $(\mathbb{P}_2(Y),G_{Y}^{\#})$, where  $\tilde{\epsilon}$ tends to $0$ as $\epsilon$ tends to $0$.
\end{lemma}

\begin{proof}
Define $\theta_{\#}$ and $\psi_{\#}$ by the rules $\theta_{\#}(g_{\#})=(\theta(g))_{\#}$  and $\psi_{\#}(\lambda_{\#})=(\psi(\lambda))_{\#}$. By Corollary 4.3 in \cite{LottandVillani}, the map $f_{\#}$ is an $\tilde{\epsilon}$-approximation, where 
\[\tilde{\epsilon}=6\epsilon+\sqrt{\epsilon(2\diam(Y)+\epsilon)}.\]

Let $\mu\in \mathbb{P}_2(X)$.  Then,
\[
W_2^2(\theta_{\#}(g_{\#})(f_{\#}\mu),f_{\#}(g_{\#}\mu))\leq \int d^2(\theta(g)(f(x)),f(gx)) d\mu \leq \epsilon^2.
\]
The last inequality is similar.
\end{proof}

As a corollary, we have the following result.

\begin{corollary}
Let $X_i$ and $X$ be compact metric spaces. If $(X_i,G_i)\rightarrow_{eGH}(X,G)$, then \linebreak $(\mathbb{P}_2(X_i),G_i^{\#})\rightarrow_{eGH}(\mathbb{P}_2(X),G^{\#})$.
\end{corollary}

We will also need the following lemma, an equivariant analogue of Proposition \ref{Jaime's Result}.

\begin{lemma}
\label{equivariant isometry induced from wasserstein}

Let $X$ and $Y$ be positively curved closed  Riemannian manifolds and let $G_{X}$ and $G_{Y}$ be closed subgroups of $\Isom(X)$ and $\Isom(Y)$, respectively. If $(\mathbb{P}_2(X),G_{X}^{\#})$ and $(\mathbb{P}_2(Y),G_{Y}^{\#})$ are equivariantly isometric then so are $(X,G_{X})$ and $(Y,G_{Y})$.
\end{lemma}

\begin{proof}
Let $\Phi\colon \mathbb{P}_2(X)\rightarrow \mathbb{P}_2(Y)$ be an isometry and let $\Theta:G_{X}^{\#}\rightarrow G_{Y}^{\#}$ be an isomorphism such that for any $g\in \Isom(X)$, one has $\Phi\circ g_{\#}=\Theta(g_{\#})\circ \Phi$. Then, from Corollary 3.8 in \cite{santos2021isometric}, it follows that $\Phi$ sends Dirac deltas onto Dirac deltas. Thus, there exists an isometry $F\colon X\rightarrow Y$ such that $F_{\#}$ and $\Phi$ agree on Dirac deltas. Moreover, there exists an isomorphism $\theta\colon G_{X}\rightarrow G_{Y}$ such that for $g\in G_{X}$, $\theta(g)_{\#}=\Theta(g_{\#})$. Thus it follows that $F\circ g=\theta(g)\circ F$.
\end{proof}

We are now ready to prove theorem \ref{theorem E}.

\subsection*{Proof of Theorem \ref{theorem E}}
The implication $(2)\implies (1)$ follows from Lemma \ref{pro: inducing-equivariant-convergence-wasserstein-spaces}. The implication $(2)\implies (3)$ is trivial. Now we shall verify the implication $(1)\implies (2).$ Since $X_i$ and $X$ are isometrically rigid, there exists unique closed subgroups $G_i$ of $\Isom(X_i)$ and $G$ of $\Isom(X)$ such that $G_i^{\#}=H_i$ and $G^{\#}=H$. It now remains to verify that $(X_i,G_i)\rightarrow_{eGH}(X,G)$. As usual, assume otherwise. So, there exists a $\delta>0$ such that, up to a subsequence, $d_{eGH}((X_i,G_i),(X,G))\geq \delta$ for all $i$. Here, $d_{eGH}$ denotes the equivariant Gromov--Hausdorff distance. Since $\{X_i\}_{i}$ forms a family of compact Alexandrov spaces with curvature uniformly bounded below, dimension uniformly bounded above, and uniform upper diameter bound, it follows that there exists a compact Alexandrov space $X'$, such that, up to a subsequence, $X_i\rightarrow_{GH}X'$. Up to a further subsequence, there exists a closed subgroup $G'$ of $\Isom(X')$ such that $(X_i,G_i)\rightarrow_{eGH}(X',G')$. Hence, by Lemma \ref{pro: inducing-equivariant-convergence-wasserstein-spaces}, it follows that $(\mathbb{P}_2(X_i),H_i)\rightarrow_{eGH}(\mathbb{P}_2(X'),(G')^{\#})$. Therefore, $(\mathbb{P}_2(X'),(G')^{\#})$ and $(\mathbb{P}_2(X),G^{\#})$ are equivariantly isometric. The previous lemma shows that $(X',G')$ and $(X,G)$ are equivariantly isometric, which is a contradiction.
\qed

\subsection*{Proof of Corollary \ref{corollary 3}}
Theorem \ref{theorem E} ensures there exists closed subgroups $G_i$ and $G$ such that $G_i^{\#}=H_i$, $G^{\#}=H$ and such that $(X_i,G_i)\rightarrow_{eGH}(X,G)$. Note that the isometric rigidity of the $X_i$ and $X$ ensures that $G_i$ and $G$ have the same dimension as $H_i$ and $H$. Hence the result follows from Theorem A in  \cite{alattar2023stability}.
\qed

\printbibliography

\end{document}